\documentclass[12pt]{article}

\usepackage{url}
\usepackage{amsmath,amsthm,amssymb,amsfonts,graphicx,enumerate}
\usepackage{tikz}
\usepackage{float}
\usepackage{caption}
\usepackage{subcaption}
\usepackage{enumitem}
\setlist[enumerate]{leftmargin=.5in}
\setlist[itemize]{leftmargin=.5in}

\newtheorem{theorem}{Theorem}[section]

\newtheorem{lemma}[theorem]{Lemma}
\newtheorem{corollary}[theorem]{Corollary}
\newtheorem{proposition}[theorem]{Proposition}
\theoremstyle{definition}
\newtheorem{definition}[theorem]{Definition}
\newtheorem{example}[theorem]{Example}

\theoremstyle{remark}
\newtheorem{remark}[theorem]{Remark}

\newcommand{\rank}{\textup{rank}\,}
\newcommand{\rankp}{\textup{rank}_\theta \,}
\newcommand{\ranks}{\textup{rank}_{\pm} \,}

\newcommand{\rankpsdr}{\textup{rank}_{\text{psd}}\,}
\newcommand{\rankpsd}{\textup{rank}^{\mathbb{C}}_{\text{psd}}\,}

\newcommand{\sqrtrank}{\textup{rank}_{\! \! {\sqrt{\ }}}\,}
\renewcommand{\S}{\mathcal S}
\newcommand{\CC}{\mathbb C}
\newcommand{\RR}{\mathbb R}
\newcommand{\Aalg}{\mathcal{A}_{\textup{alg}}}
\newcommand{\A}{\mathcal{A}}

\title{The phaseless rank of a matrix}

\author{Ant\'onio Pedro Goucha \and Jo\~{a}o Gouveia  \thanks {
The first author's research was supported through a PhD scholarship from FCT, grant PD/BD/135276/2017. This work was supported by the Centre for Mathematics of the University of Coimbra--UID/MAT/00324/2019, funded by the Portuguese Government through FCT/MEC and co-funded by the European Regional Development Fund through the Partnership Agreement PT2020.}}
\date{}

\begin{document}

\maketitle

\begin{abstract}
We consider the problem of finding the smallest rank of a complex matrix whose absolute values of the entries are given. We call this minimum the phaseless rank of the matrix of the entrywise absolute values. In this paper we study this quantity, extending a classic result of Camion and Hoffman and connecting it to the study of amoebas of determinantal varieties and of semidefinite representations of convex sets. As a consequence, we prove that the set of maximal minors of a matrix of indeterminates form an amoeba basis for the ideal they define, and we attain a new upper bound on the complex semidefinite extension complexity of polytopes, dependent only on their number of vertices and facets. We also highlight the connections between the notion of phaseless rank and the problem of finding large sets of complex equiangular lines or mutually unbiased bases.
\end{abstract}

\section{Introduction}
In this paper we study a basic optimization problem: given the absolute values of the entries of a complex matrix, what is the smallest rank that it can have.
In other words, we want the solution to the rank minimization problem for a matrix under complete phase uncertainty. This defines a natural quantity that we will associate to
the matrix of absolute values and  call the \emph{phaseless rank} of the matrix.

\begin{definition}
Given $A \in \mathbb{R}^{n\times m}_+$, the set of matrices equimodular with $A$ is denoted by $$\Omega(A) = \{B \in \mathbb{C}^{n\times m}: |B|=A \text{ i.e., } |B_{ij}|=A_{ij}, \forall i,j \}$$
and its phaseless rank is defined as
$$\rankp(A) = \min \{\rank(B): B \in \Omega(A) \}.$$
\end{definition}

Equivalently, the phaseless rank of $A \in \mathbb{R}^{n\times m}_+$ can be written as $$\rankp(A) = \min \{\rank(A \circ B): B \in \mathbb{C}^{n\times m},\ |B_{ij}|=1, \forall i,j \},$$ where $\circ$ represents the Hadamard product of matrices. It is obvious that $\rankp(A) \leq \rank(A)$, and it is not hard to see that we can have a strict inequality.

\begin{example} \label{ex:gapphaseless}
Consider the $4\times 4$ derangement matrix,
$$D_4=\begin{bmatrix}
0 & 1 & 1 & 1\\
1 & 0 & 1 & 1\\
1 & 1 & 0 & 1\\
1 & 1 & 1 & 0
\end{bmatrix}.$$
We have $\rank(D_4)=4$ and, for any real $\theta$, the matrix
$$\begin{bmatrix}
0 & 1 & 1 & 1\\
1 & 0 & e^{i(\theta + \pi)} & e^{i(\theta + \frac{2\pi}{3})}\\
1 & e^{i\theta} & 0 & e^{i(\theta + \frac{\pi}{3})}\\
1 & e^{i(\theta - \frac{\pi}{3})} & e^{i(\theta - \frac{2\pi}{3})} & 0
\end{bmatrix}$$
has rank $2$. Since this matrix has as entrywise absolute values the entries of $D_4$, $\rankp(D_4) \leq 2$, and in fact we have equality.
With some extra effort one can show that up to row and column multiplication by complex scalars of absolute value one, and conjugation, this is the
only element in the equimodular class of $D_4$ with rank less or equal than two.
\end{example}

The study of this quantity can be traced back to \cite{camion1966nonsingularity}, where the problem of characterizing $A \in \mathbb{R}^{n\times n}_+$ for which we have $\rankp(A)=n$ is solved. In that paper, the question is seen as finding a converse for the diagonal dominance, a sufficient condition for nonsingularity of a matrix. This result was further generalized in \cite{levinger1972generalization}, where a lower bound is derived for $\rankp(A)$ for general $A$, and some special cases are studied, although the rank itself is never formally introduced. While the result of Camion and Hoffman is well known, there was little, if any, further developments in minimizing the rank over an equimodular class. This problem has, however, resurfaced in recent years under different guises in both the theory of semidefinite lifts of polytopes and amoebas of algebraic varieties. In this work we build on the work of these foundational papers, deriving some new results and highlighting the consequences they have in those related areas.

The paper is organized as follows. In the next section we introduce formally the notions of phaseless and signless ranks and show some relations between them and other rank notions found in the literature. In Section 3, we relate the notion of phaseless rank with questions in amoeba theory and semidefinite representability of sets, providing motivation and intuition to what follows. In Section 4 we revisit a result of Camion and Hoffman, reproving it in a language well-suited to our needs, and drawing some simple consequences. Section 5 covers our extensions and complements to this classic result. Finally, in Section 6, we draw implications from those results to those of the connecting areas. Those include proving that the maximal minors form an amoeba basis for the variety they generate and giving an explicit semialgebraic description for those amoebas, as well as deriving a new upper bound for the complex semidefinite rank of polytopes in terms of their number of facets and vertices, and connecting the notion of phaseless rank to the problem of finding large sets of complex equiangular lines or mutually unbiased bases.

\section{Notation, definitions and basic properties} \label{sec:definitions}

Throughout these notes we will use $\mathbb{R}^{n\times m}_+$ and $\mathbb{R}^{n\times m}_{++}$ to denote the sets of $n\times m$ real matrices with nonnegative and positive entries, respectively. We will also use $\S^n$, $\S_+^n$, $\S^n(\CC)$ and $\S^n_+(\CC)$ to denote, in this order, the sets of $n \times n$ real symmetric matrices, $n \times n$ real positive semidefinite matrices, $n \times n$ complex hermitian matrices and $n \times n$ complex positive semidefinite matrices. Given a matrix in $\mathbb{R}^{n\times m}_+$, we defined its phaseless rank as the smallest rank of a complex matrix equimodular with it. If we restrict ourselves to the real case, we still obtain
a sensible definition, and we will denote that quantity by \emph{signless rank}.

\begin{definition}
Let $A\in \mathbb{R}^{n\times m}_+$.
$$\ranks(A) = \min \{\rank(B): B \in \Omega(A) \cap \mathbb{R}^{n\times m} \}.$$
\end{definition}

Equivalently, this amounts to minimizing the rank over all possible sign attributions to the entries of $A$.
By construction, it is clear that $\rankp(A) \leq \ranks(A) \leq \rank(A)$ for any nonnegative matrix $A$ and all inequalities can be strict.

\begin{example}\label{ex:signlessgap}
Let us revisit Example \ref{ex:gapphaseless}, and note that the signless rank of $D_4$ is $4$. Indeed, if we expand the determinant of that matrix, we get an odd number of nonzero terms, all $1$ or $-1$, so no possible sign attribution can ever make it sum to zero. Thus, $\rankp(D_4)<\ranks(D_4)=\rank(D_4)$. On the other hand, if we consider matrix
$$B=\begin{bmatrix}
2 & 1 & 1\\ 1 & 2 & 1 \\ 1 & 1 & 2
\end{bmatrix}$$
it is easy to see that $\rank(B)=3$ but that flipping the signs of all the $1$'s to $-1$'s drops the rank to $2$, as the matrix rows will then sum to zero, so we have $\rankp(B)=\ranks(B)<\rank(B)$. If we want all inequalities to be strict simultaneously, it is enough to make a new matrix with $D_4$ and $B$ as its diagonal blocks.
\end{example}

A short remark at the end of \cite{camion1966nonsingularity} points to the fact that the problem seems much harder over the reals, due to the combinatorial nature it assumes in that context. In fact, the signless rank is essentially equivalent to a different quantity, introduced in \cite{gouveia2013polytopes}, denoted by the \emph{square root rank} of a nonnegative matrix. In fact, by definition, $\ranks(A)=\sqrtrank(A \circ A)$ or, equivalently, $\sqrtrank(A)=\ranks(\sqrt[\circ]{A})$, where $\circ$ is the Hadamard product and $\sqrt[\circ]{A}$ is the Hadamard square root of $A$. As such, the complexity results proved in \cite{fawzi2015positive} for the square root rank still apply to the signless rank, implying the NP-hardness of the decision problem of checking if an $n \times n$ nonnegative matrix has signless rank equal to $n$. The proof of that complexity result relies on the combinatorial nature of the signless rank and fails in the more continuous notion of phaseless rank (in fact we will see the analogous result to be false for the phaseless rank), offering some hope that this later quantity will prove to be easier to work with. We will focus most of our attention in this latter notion.

The connection to the square root rank can actually be used to derive some lower bounds for both $\ranks$ and $\rankp$.

\begin{lemma} \label{lem:inequality}
Let $A \in  \mathbb{R}^{n\times m}_+$ and $r=\rank(A \circ A)$. Then, $\ranks(A) \geq \frac{\sqrt{1+8r}-1}{2}$ and $\rankp(A) \geq \sqrt{r}$.
\end{lemma}
\begin{proof}
The basic idea is that if we take a matrix $B$ equimodular with $A$ and a minimal factorization $B=UV^t$, and let $u_i$ and $v_j$ be the $i$-th and $j$-th rows of $U$ and $V$, respectively, we have
$$\langle u_i u_i^* , v_j v_j^* \rangle = |\langle u_i, v_j \rangle|^2 = |b_{ij}|^2 = a_{ij}^2 .$$
Now all the $u_i u_i^*$ and $v_j v_j^*$ come from the space of real symmetric matrices of size $\ranks(A)$, if we are taking real matrices $B$, and complex hermitian matrices of size $\rankp(A)$, if we are taking complex matrices $B$.
Since the real dimensions of these spaces are, respectively, $\binom{\ranks(A)+1}{2}$ and $\rankp(A)^2$, and they give real factorizations of $A \circ A$, we get the inequalities
$$\rank(A \circ A) \leq \binom{\ranks(A)+1}{2} \textrm{ \ \ \  and \ \ \ } \rank(A \circ A) \leq \rankp(A)^2,$$
which, when inverted, give us the intended inequalities.
\end{proof}

This result is known in the context of semidefinite rank, and is included here only for the purpose of a unified treatment. An additional very simple property that is worth noting is that a nonnegative matrix has rank one if and only if it has signless rank one, if and only if it has phaseless rank one. This simple fact immediately tells us that the matrices $D_4$ and $B$ in Example \ref{ex:signlessgap} have phaseless rank $2$, since we have proved it is at most $2$ and those matrices have rank greater than one.

Besides the problem of computing or bounding the phaseless rank, we will be interested in the geometry of the set of rank constrained matrices. In order to refer to them we will introduce some notation.

\begin{definition}
Given positive integers $k,n$ and $m$ we define the following subsets of  $\mathbb{R}^{n\times m}_+$:
$$P^{n\times m}_{k} = \{ A\in \mathbb{R}^{n\times m}_+: \rankp(A)\leq k \},$$
$$S^{n\times m}_{k} = \{ A\in \mathbb{R}^{n\times m}_+: \ranks(A)\leq k \},$$
and
$$R^{n\times m}_{k} = \{ A\in \mathbb{R}^{n\times m}_+: \rank(A)\leq k \}.$$
\end{definition}

It is easy to see that these are all semialgebraic sets. Moreover, the set $R^{n\times m}_{k}$ is well understood, since it is simply the variety of matrices of rank at most $k$, defined by the $k+1$-minors, intersected with the nonnegative orthant. It is also not too hard to get a grasp on the set $S^{n\times m}_{k}$, as this is the union of the variety of matrices of rank at most $k$ with all its $2^{n \times m}$ possible reflections attained by flipping the signs of a subset of variables, intersected with the nonnegative orthant. In particular, we have a somewhat simple algebraic description of both these sets, and they have the same dimension, $k(m+n-k)$.

For $P^{n\times m}_{k}$, all these questions are much more difficult. Clearly we have $R^{n\times m}_{k} \subseteq S^{n\times m}_{k} \subseteq P^{n\times m}_{k}$, which gives us some lower bound on the dimension of the space, but not much else can be immediately derived.

The relations between all these sets are illustrated in Figure \ref{fig:sets}, where we can see a random $2$-dimensional slice of the cone of nonnegative $3\times 3$ matrices (in pink) with the corresponding slice of the region of phaseless rank at most $2$, highlighted in yellow, while the slices of the algebraic closures of the regions of signless rank at most $2$ and usual rank at most $2$ are marked in dashed and solid lines, respectively. Note that Figure \ref{fig:sets} suggests $P^{3\times 3}_{2}$ is full-dimensional. In fact, $P^{n\times n}_{k}$ is full-dimensional in $\mathbb{R}^{n\times n}_+$ for any $k \geq  \frac{n+1}{2} $. This observation follows from Corollary \ref{cor:dimension}.

\begin{figure}[H]
  \centering
    \centerline{\includegraphics[width=0.5\textwidth]{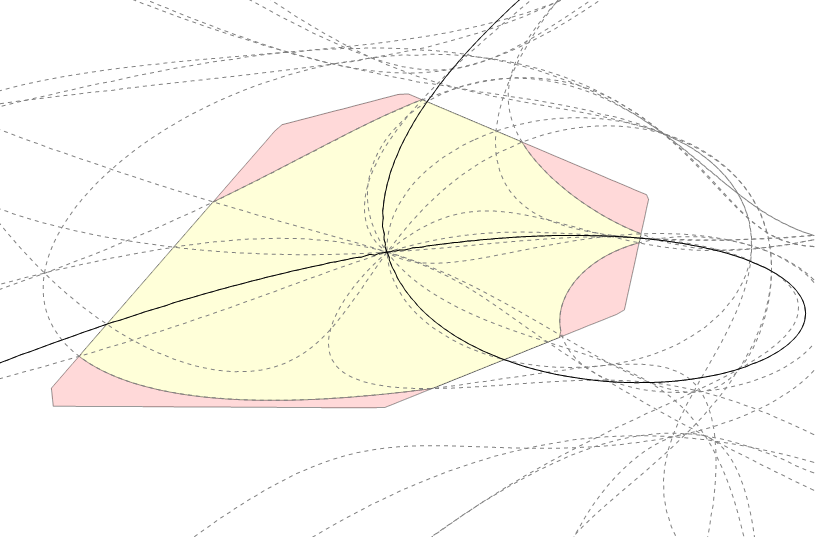}}
	\caption{Slice of the cone of nonnegative $3\times 3$ matrices with $P^{3\times 3}_{2}$, $S^{3\times 3}_{2}$ and $R^{3\times 3}_{2}$ highlighted}
	\label{fig:sets}
\end{figure}

\section{Motivation and connections}

As mentioned in the introduction, the concept of phaseless rank is intimately connected to the concept of semidefinite rank of a matrix, used, for instance, to study semidefinite representations of polytopes and amoebas of algebraic varieties. In this section we will briefly introduce each of those areas and establish the connections, as those were the motivating reasons for our study of the subject.

\subsection{Semidefinite extension complexity of a polytope}

The semidefinite rank of a matrix was introduced in \cite{gouveia2013lifts} to study the semidefinite extension complexity of a polytope. Recall that given a $d$-polytope $P$, its \emph{semidefinite extension complexity}
its the smallest $k$ for which one can find $A_0, A_1, \dots, A_m \in \S^k$ such that
$$P=\left\{ (x_1,\dots,x_d) \in \RR^d: \exists \, x_{d+1}, \dots, x_m \in \RR \textrm{ s.t. }A_0 + \sum_{i=1}^m x_i A_i \succeq 0 \right\}.$$
In other words, it is the smallest $k$ for which one can write $P$ as the projection of a slice of the cone of $k\times k$ real positive semidefinite matrices. In order to study this concept one has to introduce the notion of \emph{slack matrix} of a polytope. If $P$ is a polytope with vertices $p_1$,..., $p_v$ and facets cut out by the inequalities $\langle a_1, x\rangle \leq b_1$, ..., $\langle a_f, x\rangle \leq b_f$, then we define its slack matrix to be the nonnegative $v \times f$ matrix $S_P$ with entry $(i,j)$ given by $b_j-\langle a_j, p_i\rangle$.

Additionally, the \emph{semidefinite rank} of a nonnegative matrix $A \in \RR^{n \times m}_+$, $\rankpsdr(A)$, is the smallest $k$ for which one can find $U_1 \dots, U_n, V_1, \dots, V_m \in \S_+^k$ such that $A_{ij}=\langle U_i, V_j \rangle$. By the main result in  \cite{gouveia2013lifts} one can characterize the semidefinite extension complexity of a $d$-polytope $P$ in terms of the semidefinite rank of its slack matrix.

\begin{proposition}
The extension complexity of a polytope $P$ is the same as the semidefinite rank of its slack matrix, $\rankpsdr(S_P)$.
\end{proposition}

For a thorough treatment of the positive semidefinite rank, see \cite{fawzi2015positive}. As noted in \cite{goucha2017ranks,lee2017some}, one can replace real positive semidefinite matrices with complex positive semidefinite matrices and everything still follows through. More precisely, if one defines the \emph{complex semidefinite extension complexity} of $P$ as the smallest $k$ for which  one can find $B_0, B_1, \dots, B_m \in \S^k(\CC)$ such that
$$P=\left\{ (x_1,\dots,x_d) \in \RR^d: \exists \, x_{d+1}, \dots, x_m \in \RR \textrm{ s.t. } B_0 + \sum_{i=1}^m x_i B_i \succeq 0 \right\},$$
and the \emph{complex semidefinite rank} of a matrix $A \in \RR^{n \times m}_+$, $\rankpsd(A)$, as the smallest $k$ for which one can find $U_1 \dots, U_n, V_1, \dots, V_m \in \S_+^k(\CC)$ such that $A_{ij}=\langle U_i, V_j \rangle$, the analogous of the previous proposition still holds.

\begin{proposition}
The complex extension complexity of a polytope $P$ is the same as the complex semidefinite rank of its slack matrix, $\rankpsd(S_P)$.
\end{proposition}

The study of the semidefinite extension complexity of polytopes has seen several important recent breakthroughs, and has brought light to this notion of semidefinite rank. It turns out that the notions of signless and phaseless rank give a natural upper bound for these quantities.

\begin{proposition}[\cite{fawzi2015positive,lee2017some}]\label{prop:ineqs}
Given a nonnegative matrix $A$, we have $\rankpsd(A) \leq \rankp(\sqrt[\circ]{A})$ and $\rankpsdr(A) \leq \ranks(\sqrt[\circ]{A})$.
\end{proposition}

The proof of this result is essentially the one we used in Lemma \ref{lem:inequality}, as factorizations of an equimodular matrix with $\sqrt[\circ]{A}$ give rise to semidefinite factorizations to $A$ by taking outer products of the rows of the factors. This bound is particularly important in the study of polytopes, since it fully characterizes polytopes with minimal extension complexity.

\begin{proposition}[\cite{gouveia2013polytopes,goucha2017ranks}]
Given a $d$-polytope $P$, we have that its complex and real semidefinite complexities are at least $d+1$. Moreover, they are $d+1$ if and only if $\rankp(\sqrt[\circ]{S_P})=d+1$ or $\ranks(\sqrt[\circ]{S_P})=d+1$, respectively.
\end{proposition}

This fact allowed to characterize minimally sdp-representable polytopes in $\RR^3$ and $\RR^4$ in the real case (see \cite{gouveia2013polytopes,gouveia2017fourdim}) and has given some interesting consequences for the complex case (see \cite{goucha2017ranks}). One of the main motivations for us to study the phaseless rank comes precisely from this connection.

\subsection{Amoebas of determinantal varieties}

Another way of looking at phaseless rank is through amoeba theory. Amoebas are geometric objects that were introduced by Gelfand, Kapranov and Zelevinsky in \cite{gelfand2008discriminants} to study algebraic varieties. These complex analysis objects have applications in algebraic geometry, both complex and tropical, but are notoriously hard to work with. They are the image of a variety under the entrywise logarithm of the absolute values of the coordinates.
\begin{definition}
Given a complex variety $V \subseteq \mathbb{C}^{n}$, its \textit{amoeba} is defined as
$$\A(V)=\{ \text{Log}|z|=(\log|z_1|,\ldots,\log|z_n|):z\in V \cap {(\mathbb{C}^{*})}^{n}\}.$$
\end{definition}
Deciding if a point is on the amoeba of a given variety, the so called \emph{amoeba membership problem}, is notoriously hard,  making even the simple act of drawing an amoeba a definitely nontrivial task. Other questions like computing volumes or even dimensions of amoebas are also hard. A slightly more algebraic version of this object can be defined by simply taking the entrywise absolute values, and omitting the logarithm.
\begin{definition}
Given a complex variety $V \subseteq \mathbb{C}^{n}$, its \textit{algebraic} or \textit{unlog amoeba} is defined as
$$\Aalg(V)=\{|z|=(|z_1|,\ldots,|z_n|):z\in V \}.$$
\end{definition}
Considering this definition, it is clear how it relates to the notion of phaseless rank by way of \emph{determinantal varieties}. These and their corresponding ideals are a central object in both commutative algebra and algebraic geometry, and a great volume of research has been focused on studying them. Given positive integers $n,m$ and $k$, with $k \leq \min\{n,m\}$, we define the determinantal variety $Y_{k}^{n,m}$ as the set
of all $n\times m$ complex matrices of rank at most $k$. It is clear that this is simply the variety associated to $I^{n,m}_{k+1}$, the ideal of the $k+1$ minors of an $n\times m$ matrix with distinct variables as entries.


\begin{example} \label{ex:amoeba}
In Figure \ref{fig:amb} we consider the amoeba of the variety $V$ defined by the following $3 \times 3$ determinant:
$$\det \begin{bmatrix} 1 & x & y \\ x & 1 & z \\ y & 0 & 1 \end{bmatrix}=1-x^2+xyz-y^2=0.$$
\begin{figure}[H]
  \centering
    \centerline{\includegraphics[width=0.4\textwidth]{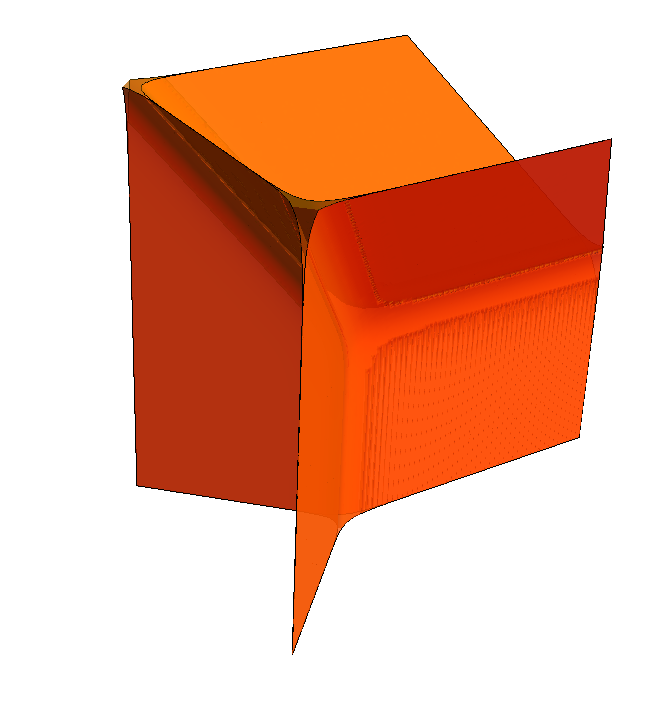} \hspace{2cm} \includegraphics[width=0.4\textwidth]{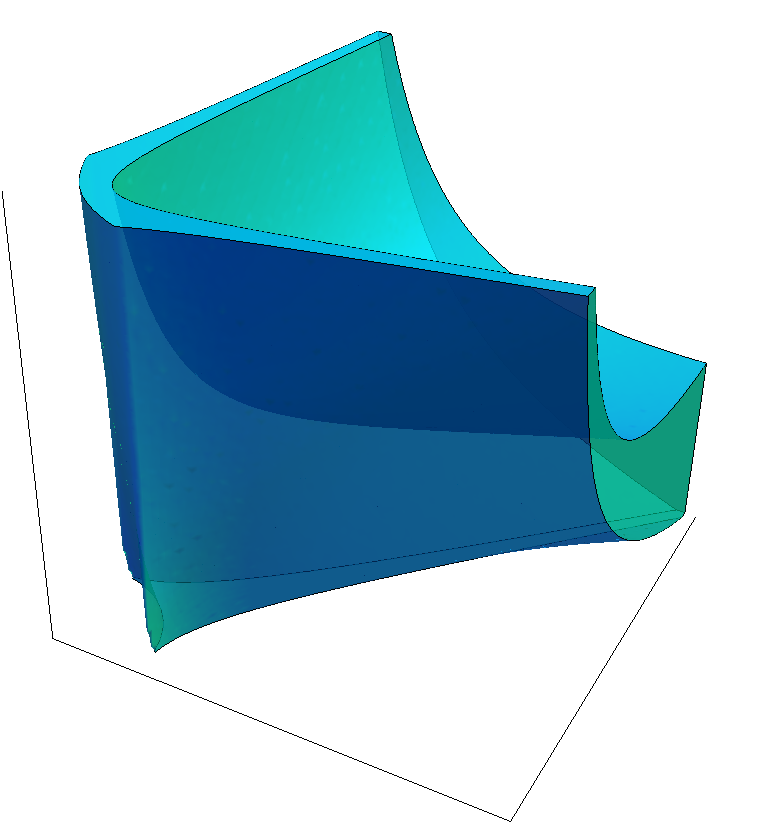}}
	\caption{$\A(V)$ and $\Aalg(V)$ of a determinantal variety.}
	\label{fig:amb}
\end{figure}
\end{example}

Note that directly from the definition of amoeba, we have that the locus of $n\times m$ matrices of phaseless rank at most $k$ is an algebraic amoeba of a determinantal variety, more precisely,
$$P^{n\times m}_{k} = \Aalg(Y_{k}^{n,m}).$$

\begin{example}
The blue region in Example \ref{ex:amoeba} is exactly the region of the values of $x, y$ and $z$ for which
$$\rankp  \begin{bmatrix} 1 & x & y \\ x & 1 & z \\ y & 0 & 1 \end{bmatrix} \leq 2.$$
This is not totally immediate, since in the phaseless rank definition we are allowed to freely choose a phase independently to each entry of the matrix, which includes the $1$'s and also the possibility of different phases for different copies of the same variable, which is not allowed in the amoeba definition. However, since multiplying rows and columns by unitary complex numbers does not change absolute values or rank, we can make any phase attribution into one of the right type, and the regions do coincide.
\end{example}


More generally, computing the phaseless rank of a matrix corresponds essentially to solving the membership problem in the determinantal amoeba, so any result on the phaseless rank can immediately be interpreted as a result about this fundamental object in amoeba theory. Also on the interconnectedness between amoebas and phaseless rank, see Proposition 5.2 from \cite{forsberg2000laurent}, which, in our language, states that the intersection of a fixed number of compactified hyperplane amoebas is empty if and only if the phaseless rank of a specific nonnegative matrix is maximal.

\section{Camion-Hoffman's Theorem}

In this section we set to revisit Camion-Hoffman's Theorem, originally proved in \cite{camion1966nonsingularity}. The main purpose of this section is to set the ideas behind this result in a language and generality that will be convenient for our goals, highlighting the facts that will be most useful, and introducing the necessary notation. For the sake of completeness a proof of the theorem is included. The main idea behind the proof is the simple observation that checking for nonmaximal phaseless rank is simply a linear programming feasibility problem, i.e., checking if a nonnegative matrix has nonmaximal phaseless rank amounts to checking if a specific polytope is nonempty. Here, by nonmaximal phaseless rank we mean that the phaseless rank is less than the minimum of the matrix dimensions.

Inspired by the language of amoeba theory (\cite{purbhoo2008nullstellensatz}) we introduce the notion of \emph{lopsidedness}. Simply put, a list of nonnegative numbers is lopsided if one is greater than the sum of all others.
It is easy to see geometrically, that a nonlopsided list of numbers can always be realized as the lengths of the sides of a polygon in $\RR^2$. Interpreting it in terms of complex numbers we get that a list of nonnegative real numbers $\{a_1,\dots,a_n\}$ is nonlopsided if and only if there are $\theta_k \in [0,2\pi]$ for which $\sum_{k=1}^n a_k e^{\theta_k i} = 0$. This is enough to give us a first characterization of nonmaximal phase rank.

\begin{lemma}\label{lem:lop}
Let $A \in \mathbb{R}^{n\times m}_+$, with $n\leq m$. Then, $\rankp(A)<n$ if and only if there is $\lambda \in \mathbb{R}^n_+$ with $\sum_{i=1}^{n} \lambda_i = 1$ such that, for $l=1,\ldots,m$, $\{A_{1l}\lambda_1,\ldots,A_{nl}\lambda_n\}$ is not lopsided.
\end{lemma}
\begin{proof}
First note that $\rankp(A)<n$ if and only if there exists a matrix $B$ with $B_{kl}=A_{kl} e^{i\theta_{kl}}$ for all $k,l$, such that $\rank(B)<n$. This is the same as saying that the rows of $B$ are linearly dependent, and so there exists a nonzero complex vector $z=(z_1,\ldots,z_n)$ such that $\sum |z_j|=1$ and $\sum_{k=1}^n A_{kl} z_k e^{i\theta_{kl}}=0$, for $l=1,\ldots,m$.
By the observation above, this is equivalent to saying that, for $l=1,\ldots,m$, $\{A_{1l}|z_1|,\ldots,A_{nl}|z_n|\}$ is not lopsided.
\end{proof}

The previous result tells us essentially that $\rankp(A)<n$ if and only if we can scale rows of $A$ by nonnegative numbers in such a way that the entries on each of the columns verify the generalized triangular inequalities. The conditions for a matrix $A \in \mathbb{R}^{n\times m}_+$, with $n\leq m$, to verify $\rankp(A)<n$ can now be simply stated as checking if there exists $\lambda \in \mathbb{R}^n$ such that
$$\begin{cases}
A_{ij} \lambda_i \leq \sum_{k\neq i} A_{kj} \lambda_k, \, j=1,\ldots,m, \, i=1,\ldots,n \\ \\
\lambda_i\geq 0, \, i=1,\ldots,n \\ \\
\sum_{i=1}^{n} \lambda_i = 1.
\end{cases}$$
We have just observed the following result.

\begin{corollary}\label{cor:lp}
Given $A \in \mathbb{R}^{n\times m}_+$, with $n\leq m$, deciding if $\rankp(A)<n$ is a linear programming feasibility problem.
\end{corollary}

Note that this gives us a polynomial time algorithm (on the encoding length) for checking nonmaximality of the phaseless rank. Equivalently, this gives us a  polynomial time algorithm to solve the amoeba membership problem for the determinantal variety of maximal minors.

We are now almost ready to state and prove a version of the result of Camion-Hoffman. We need only to briefly introduce some facts about $M$-matrices.
\begin{definition}
An $n\times n$ real matrix $A$ is an M-matrix if it has nonpositive off-diagonal entries and all its eigenvalues have nonnegative real part.
\end{definition}
The class of $M$-matrices is well studied, and there are numerous equivalent characterizations for them. Of particular interest to us will be the following characterizations.

\begin{proposition} \label{prop:mmatrix}
Let $A \in \RR^{n\times n}$ have nonpositive off-diagonal entries. Then the following are equivalent.
\begin{enumerate}[label=(\roman*)]
\item $A$ is a nonsingular $M$-matrix;
\item There exists $x \geq 0$ such that $Ax > 0$;
\item The diagonal entries of $A$ are positive and there exists a diagonal matrix $D$ such that $AD$ is strictly diagonally dominant;
\item All leading principal minors are positive;
\item The diagonal entries of $A$ are positive and all leading principal minors of size at least $3$ are positive;
\item Every real eigenvalue of $A$ is positive.
\end{enumerate}
\end{proposition}

\begin{remark}
Characterizations ii, iii, iv and vi can be found in Theorem 2.3 of \cite{doi:10.1137/1.9781611971262} and v in Corollary 2.3 of \cite{poole1974survey}.
\end{remark}

Finally, recall that given $A \in \mathbb{C}^{n\times n}$, its \emph{comparison matrix}, $\mathcal{M}(A)$, is defined by $\mathcal{M}(A)_{ij} = |A_{ij}|$, if $i=j$, and  $\mathcal{M}(A)_{ij} = -|A_{ij}|$, otherwise.

\begin{theorem}[Camion-Hoffman's Theorem] \label{thm:cam_hof}
Given  $A \in \mathbb{R}^{n\times n}_{+}$, $\rankp(A)=n$ if and only if there exists a permutation matrix $P$ such that $\mathcal{M}(AP)$ is a nonsingular M-matrix.
\end{theorem}
\begin{proof}
Let the entries of $A$ be denoted by $a_{ij}$, $1 \leq i,j \leq n$. By Corollary \ref{cor:lp}, $\rankp(A)=n$, if and only if the linear problem
$$M \lambda \leq 0, \quad \lambda\geq 0, \quad \sum_{i=1}^{n} \lambda_i = 1$$
is not feasible, where
$$M=\begin{bmatrix} M_1 \\ M_2 \\ \vdots \\ M_n\end{bmatrix}, \ \ \textrm{ with } M_i=\begin{bmatrix}
a_{1i} & -a_{2i} & \ldots & -a_{ni}\\
-a_{1i} & a_{2i} & \ldots & -a_{ni}\\
\vdots & \vdots & \ddots & \vdots\\
-a_{1i} & -a_{2i} & \ldots & a_{ni}\end{bmatrix} \textrm{ for } i=1,\dots,n.
$$
By Ville's Theorem, a simple variant of Farkas' Lemma, this is equivalent to the existence of $y\geq 0$ such that $y^T M > 0$. Furthermore, since $y^T M$ is in the convex cone generated by the rows of $M$, then, by Carath\'{e}odory's Theorem, $y^T M$ can be written as a nonnegative combination of $n$ rows of $M$. Let us call $y'^T M'$ to this representation of $y^T M$, where $M'$ is a submatrix of $M$ containing exactly $n$ rows of $M$ and $y'\geq 0$.

We first observe that each column of $M'$ has exactly one nonnegative entry and all components of $y'$ should be positive.
Furthermore, if two rows of $M'$ are come from the same $M_i$, the components of $y'^T M'$ will not be all positive. So, there are $n!$ possibilities for $M'$, given by ${M'}^T=\mathcal{M}(AP)$, for some permutation matrix $P$. But then, the existence of $y'\geq 0$ such that $\mathcal{M}(AP)y'>0$ is equivalent to $\mathcal{M}(AP)$ being a nonsingular $M$-matrix by Proposition \ref{prop:mmatrix}, concluding the proof.
\end{proof}

Note that, while equivalent, this is not the original statement of Camion-Hoffman's result. This precise version can be found, for example, in \cite{BRADLEY1969105}, as a corollary of a stronger result. The way it is originally stated, Camion-Hoffman's Theorem says that, if $A$ is an $n\times n$ matrix with nonnegative entries, every complex matrix in the equimodular class of $A$, $\Omega(A)$, is nonsingular if and only if there exists a permutation matrix $P$ and a positive diagonal matrix $D$ such that $PAD$ is strictly diagonally dominant. Proposition \ref{prop:mmatrix} immediately gives us the equivalence of both statements. We also highlight Proposition 5.3 from \cite{forsberg2000laurent}, where the authors rediscover Camion-Hoffman's Theorem in an amoeba theory context.

\begin{example} \label{ex:3x3characterization}
Let us see how Camion-Hoffman's Theorem applies to a $3\times 3$ matrix. Let $X \in \RR_+^{n\times n}$ have entries $[x_{ij}]$. We want to characterize $P_2^{3 \times 3}$, that is to say, when is $\rankp(X)\leq 2$.
By Camion-Hoffman's Theorem, this happens if and only if for every permutation matrix $P \in S_3$, we have that $\mathcal{M}(XP)$ is not a nonsingular $M$-matrix. By Proposition \ref{prop:mmatrix}, checking
if $\mathcal{M}(XP)$ is a nonsingular  $M$-matrix amounts to checking if its determinant is positive (since it is a $3\times 3$ matrix).

Hence, $\rankp(X) \leq 2$ if and only if $\det(\mathcal{M}(XP)) \leq 0$ for all $P \in S_3$. There are $6$ possible matrices $P$ giving rise to $6$ inequalities. For $P$ equal to the identity, for example, we get
$$\det \begin{bmatrix} x_{11} &-x_{12}  & -x_{13} \\ -x_{21} &x_{22}  & -x_{23} \\ -x_{31} &-x_{32}  & x_{33} \\ \end{bmatrix}  \leq 0,$$
which means
$$x_{11}x_{22}x_{33}- x_{11}x_{23}x_{32} - x_{12}x_{21}x_{33}-x_{12}x_{23}x_{31} - x_{13}x_{21}x_{32} - x_{13}x_{22}x_{31} \leq 0.$$
It is not hard to check that any other $P$ will result in a similar equality, where one monomial of the terms of the expansion of the determinant of $X$ appears with a positive sign, and all others with a negative sign.
\end{example}

This can be very useful to understand the geometry of the phaseless rank, as seen in a slightly more concrete example.

\begin{example}
Building from Example \ref{ex:3x3characterization}, let us characterize the nonnegative values of $x$ and $y$ for which the circulant matrix
$$\begin{bmatrix}
1 & x & y \\ y & 1 & x \\ x & y & 1
\end{bmatrix}$$
has phaseless rank less than $3$. Computing the six polynomials determined in that example, we find that they collapse to just four distinct ones:
$$1-x^3-y^3-3 y x, \ \  -1+x^3-y^3-3 y x, \ \ -1-x^3+y^3-3 y x, \ \ -1-x^3-y^3- y x .$$
For nonnegative $x$ and $y$, the last one is always negative, so it can be ignored. Furthermore, the other three factor each into a linear term and a nonnegative quadratic term, which can also be ignored, so we are left only with the three linear inequalities
$$1-x-y \leq 0, \ \ 1+x-y \leq 0, \ \ 1-x+y \leq 0.$$
\begin{figure}[H]
  \centering
    \centerline{\includegraphics[width=0.4\textwidth]{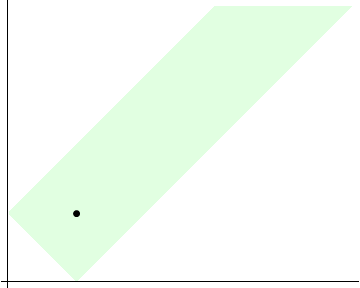}}
	\caption{Region where the $3 \times 3$ nonnegative circulant matrices have nonmaximal $\rankp$}
	\label{fig:circ}
\end{figure}
In Figure \ref{fig:circ} we can observe the region. Note that the only singular matrix in that region is that for which $x=y=1$, highlighted in the figure, every other one has usual rank equal to three. It is not hard to check that the signless rank additionally drops to two precisely on the boundary of the region.
\end{example}

\section{Consequences and extensions}

In this section, we derive some new results and strengthen some old ones, based on both Camion-Hoffman's result and, more generally, the underlying idea of using linear programming theory to study the phaseless rank.

\subsection{The rectangular case}

While we now have a full characterization for square matrices with nonmaximal phaseless rank, we are interested in extending it to more general settings. In this section we will study the case of rectangular matrices.
Note that since transposition preserves the rank, we might restrict ourselves always to the case of $A \in \RR^{n \times m}$ with $n \leq m$ for ease of notation. The simplest question one can ask is when does such a matrix have nonmaximal phaseless rank, i.e., when is $\rankp(A)<n$?

Denote by $A_I$, where $I$ is a set of $n$ distinct numbers between $1$ and $m$, the $n\times n$ submatrix of $A$ of columns indexed by elements of $I$. It is clear that if $A$ has phaseless rank less than $n$ so does $A_I$,
since the submatrices $B_I$ of a complex matrix $B$ that is equimodular with $A$ and has rank less than $n$ will be, themselves, equimodular to the matrices $A_I$ and have rank less than $n$. The reciprocal is much less clear, since the existence of singular matrices equimodular with each of the $A_I$ does not seem to imply the existence of a singular matrix globally equimodular with $A$, since patching together the phases attributions to different submatrices is not trivial. Surprisingly, the result does hold.

\begin{proposition}\label{prop:rect}
Let $A \in \mathbb{R}^{n\times m}_+$, with $n\leq m$. Then, $\rankp(A)<n$ if and only if  $\rankp(A_I)<n$ for all $I \subseteq \{1,\dots,m\}$ with $|I|=n$.
\end{proposition}
\begin{proof}
By the above discussion, the only thing that needs proof is the sufficiency of the condition $\rankp(A_I)<n$ for all $I$, since it is clearly implied by $\rankp(A)<n$. Assume that the condition holds. Then, by Lemma \ref{lem:lop}, for each  $A_I$ there exists $\lambda^I \in \mathbb{R}^n_+$ with coordinate sum one, such that for each column $l \in I$, $\{A_{1l}\lambda^{I}_1,\ldots,A_{nl}\lambda^{I}_n\}$ is not lopsided.

Given  any $x \in \mathbb{R}^n_+$, denote by $\textup{Lop}(x)$ the set of $y \in \RR^n_+$ with coordinate sum one such that $\{x_1 y_1,\ldots,x_n y_n\}$ is not lopsided. This is simply the polyhedral set
\begin{align*}
\text{Lop}(x)= \Big\{y \in \mathbb{R}^n_+, \, \sum_{i=1}^{n} y_i = 1: \, x_i y_i \leq \sum_{k\neq i} x_k y_k, \, i=1,\ldots,n \Big\}
\end{align*}
and, in particular, is convex.

Let $a_j$ denote the $j$th column of $A$. The convex sets $\text{Lop}(a_j)$, for $j=1,...,m$, are contained in the hyperplane of coordinate sum one, an $n-1$ dimensional space. Furthermore, by assumption, any $n$ of them intersect, since for any $I=\{i_1,\dots,i_n\}$, we have $\lambda^I \in \bigcap_{j\in I} \text{Lop}(a_j)$. By Helly's Theorem, we must have
$$\bigcap^m_{j=1} \text{Lop}(a_j) \neq \emptyset,$$
which means that we can take $\lambda$ in the intersection, which will then verify the conditions of Lemma \ref{lem:lop}, proving that $\rankp(A)<n$.
\end{proof}

This shows that we can reduce the $n\times m$ case to multiple $n \times n$ cases, so we can still apply Camion-Hoffman's result to study this case.

\begin{example}
Consider the family of $3\times 4$ matrices parametrized by
$$\left[
\begin{array}{cccc}
 x-y+1 & x-y+1 & x+1 & 1 \\
 1-x & -x+y+1 & 1-y & x+y+1 \\
 1-y & 1-x & 1 & x-y+1 \\
\end{array}
\right].$$
If we want to study the region where the phaseless rank is two, it is enough to look at the four $3\times 3$ submatrices and use the result of Example \ref{ex:3x3characterization} to compute the region for each of them, which are shown in Figure \ref{fig:submatrices}. The red pentagonal region is the region where the matrix is nonnegative, while the colored region inside is the region of nonmaximal rank for each of the submatrices.
\begin{figure}[H]
  \centering
    \centerline{\includegraphics[width=0.22\textwidth]{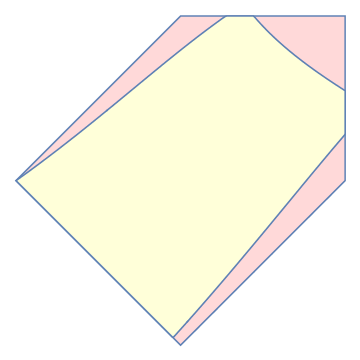} \hfill
    \includegraphics[width=0.22\textwidth]{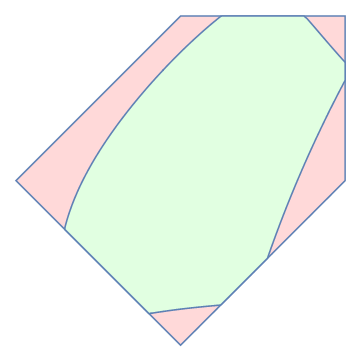}\hfill
    \includegraphics[width=0.22\textwidth]{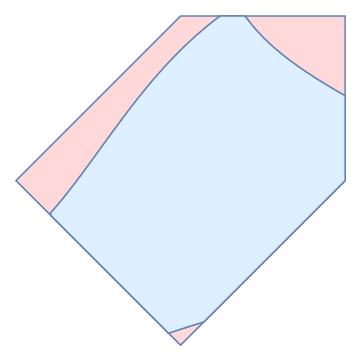}\hfill
    \includegraphics[width=0.22\textwidth]{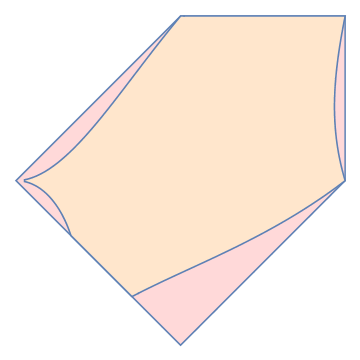}
    }
	\caption{Region of nonmaximal phase rank for each $3 \times 3$ submatrix}
	\label{fig:submatrices}
\end{figure}
By Proposition \ref{prop:rect} we then can simply intersect the four regions to observe the region where the phaseless rank of the full matrix is at most $2$. The result is shown in Figure \ref{fig:fullmatrix}
\begin{figure}[H]
  \centering
    \centerline{\includegraphics[width=0.3\textwidth]{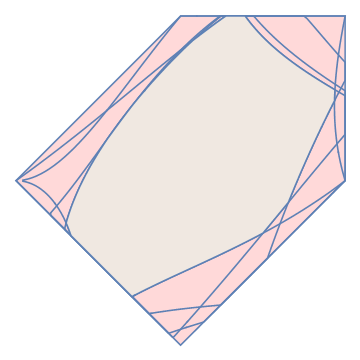}
    }
	\caption{Region of nonmaximal phase rank for the full matrix}
	\label{fig:fullmatrix}
\end{figure}
\end{example}

\subsection{Geometric implications} \label{ssec:geometric}

From Camion-Hoffman's Theorem and Proposition \ref{prop:rect} one can also derive results on the geometry of the sets $P_{n-1}^{n \times m}$, of the $n\times m$ matrices of nonmaximal phaseless rank. More precisely, we are interested in the semialgebraic descriptions of such sets, and their boundaries.

Recall that $P_k^{n \times m}$ is always semialgebraic by the Tarski-Seidenberg principle, since it is the projection of a semialgebraic set. However the description can in principle be very complicated. For this special case, Theorem \ref{thm:cam_hof} together with Proposition \ref{prop:mmatrix} give a concrete semialgebraic description of $P_{n-1}^{n \times n}$.
Recall that Theorem \ref{thm:cam_hof} states that
$$P^{n\times n}_{n-1}=\bigcap_{P \in S_n} \{ A \in \RR_+^{n\times n} : \mathcal{M}({AP}) \textrm{ is not a nonsigular } M\textrm{-matrix}\}.$$
Let $\det_i(X)$ denote the $i$-th leading principal minor of matrix $X$. The characterizations of $M$-matrices given in Proposition \ref{prop:mmatrix} then allow us to write this more concretely as
$$P^{n\times n}_{n-1}=\bigcap_{P \in S_n} \bigcup_{i=3}^n \{ A \in \RR_+^{n\times n} : \textup{det}_i(\mathcal{M}({AP})) \leq 0 \},$$
which is a closed semialgebraic set, but not necessarily basic. For the $n \times m$ case, we just have to intersect the sets corresponding to each of the $n\times n$ submatrices, so we can still write $P^{n\times m}_{n-1}$ explicitly as an intersection of unions of sets described by a single polynomial inequality.

Note that when $n=3$ the unions have a single element, which trivially gives us the following corollary.
\begin{corollary}\label{cor:3by3}
The set $P^{3 \times m}_2$ is a basic closed semialgebraic set, for $m\geq 3$.
\end{corollary}
It is generally not true that we can ignore the size $3$ minor when testing a matrix for the property of being a nonsingular $M$-matrix. However, in our particular application we can get a little more in this direction.

\begin{corollary}\label{cor:34}
For any $A \in \mathbb{R}^{4\times 4}_{+}$, we have $\rankp(A)<4$ if and only if $\det(\mathcal{M}(AP))\leq 0$  for all permutation matrices $P \in S_4$. In particular, $P^{4 \times m}_3$ is a basic closed semialgebraic set for all $m \geq 4$.
\end{corollary}

\begin{proof}

By Theorem \ref{thm:cam_hof}, $\rankp(A)=4$ if and only if, for some $P$, $\mathcal{M}(AP)$ is a nonsingular M-matrix, which implies, by Proposition \ref{prop:mmatrix}, that all its leading principal minors are positive, including its determinant.
This shows that if $\det(\mathcal{M}(AP)) \leq 0$ for all permutation matrices $P$ then  $\rankp(A)<4$.

Suppose now that $\det(\mathcal{M}(AP))>0$, for some $P$. We have to show that that this implies $\rankp(A)=4$. There exist three different permutation matrices $P_1$, $P_2$ and $P_3$, distinct from $P$ such that $$\det(\mathcal{M}(AP_1))=\det(\mathcal{M}(AP_2))=\det(\mathcal{M}(AP_3))=\det(\mathcal{M}(AP))>0.$$
Namely, $P_1$, $P_2$ and $P_3$ are obtained from $P$ by partitioning its columns in two pairs and transposing the columns in each pair. If we denote the entries of $AP$ by $b_{ij}$, $i,j \in \{1,2,3,4\}$, we get the four matrices
$\mathcal{M}(AP), \mathcal{M}(AP_1), \mathcal{M}(AP_2)$ and $\mathcal{M}(AP_3)$
as presented below in order:
$$\left[
\begin{array}{cccc}
 b_{11} & -b_{12} & -b_{13} & -b_{14} \\
 -b_{21} & b_{22} & -b_{23} & -b_{24} \\
 -b_{31} & -b_{32} & b_{33} & -b_{34} \\
 -b_{41} & -b_{42} & -b_{43} & b_{44} \\
\end{array}
\right], \ \ \ \left[
\begin{array}{cccc}
 b_{12} & -b_{11} & -b_{14} & -b_{13} \\
 -b_{22} & b_{21} & -b_{24} & -b_{23} \\
 -b_{32} & -b_{31} & b_{34} & -b_{33} \\
 -b_{42} & -b_{41} & -b_{44} & b_{43} \\
\end{array}
\right],$$
$$
\left[
\begin{array}{cccc}
 b_{13} & -b_{14} & -b_{11} & -b_{12} \\
 -b_{23} & b_{24} & -b_{21} & -b_{22} \\
 -b_{33} & -b_{34} & b_{31} & -b_{32} \\
 -b_{43} & -b_{44} & -b_{41} & b_{42} \\
\end{array}
\right], \ \ \
\left[
\begin{array}{cccc}
 b_{14} & -b_{13} & -b_{12} & -b_{11} \\
 -b_{24} & b_{23} & -b_{22} & -b_{21} \\
 -b_{34} & -b_{33} & b_{32} & -b_{31} \\
 -b_{44} & -b_{43} & -b_{42} & b_{41} \\
\end{array}
\right].$$
One can now easily check that $\det(\mathcal{M}(AP))$ can be written as
$$b_{41}\textup{det}_3(\mathcal{M}(AP_3))+b_{42}\textup{det}_3(\mathcal{M}(AP_2))+b_{43}\textup{det}_3(\mathcal{M}(AP_1))+b_{44}\textup{det}_3(\mathcal{M}(AP)),$$
which, since all $b_{ij}$ are nonnegative, means that at least one of the size
$3$ leading principal minors must be positive. By Proposition \ref{prop:mmatrix}, the corresponding matrix must be a nonsingular $M$-matrix, since it has both the $3\times 3$ and the $4 \times 4$ leading principal minors positive.

This shows that if $\det(\mathcal{M}(AP))>0$ for some permutation matrix, then Camion-Hoffman's Theorem guarantees that $\rankp(A) = 4$, completing the proof.
\end{proof}

\begin{remark}
One can extract a little more information from the proof of Corollary \ref{cor:34}.
For checking whether a $4 \times 4$ nonnegative matrix $A$ has phaseless rank less than four, we just need to check $\det \mathcal{M}(AP) \leq 0$ for all permutation matrices $P$. In addition, we also know that each determinant is obtained from four different permutation matrices, leaving only six polynomial inequalities to check.

More concretely, if $A$ has entries $a_{ij}$, and $\textup{perm}(A)$ denotes the permanent of $A$, we just have to consider the inequalities:
$$2 \left(a_{12} a_{23} a_{34} a_{41}+a_{11} a_{24} a_{33} a_{42}+a_{14} a_{21} a_{32} a_{43}+a_{13} a_{22} a_{31} a_{44}\right)-\textup{perm}(A) \leq 0,$$
$$2 \left(a_{13} a_{22} a_{34} a_{41}+a_{14} a_{21} a_{33} a_{42}+a_{11} a_{24} a_{32} a_{43}+a_{12} a_{23} a_{31} a_{44}\right)-\textup{perm}(A) \leq 0,$$
$$2 \left(a_{12} a_{24} a_{33} a_{41}+a_{11} a_{23} a_{34} a_{42}+a_{14} a_{22} a_{31} a_{43}+a_{13} a_{21} a_{32} a_{44}\right)-\textup{perm}(A) \leq 0,$$
   $$2 \left(a_{14} a_{22} a_{33} a_{41}+a_{13} a_{21} a_{34} a_{42}+a_{12} a_{24} a_{31} a_{43}+a_{11} a_{23} a_{32} a_{44}\right)-\textup{perm}(A) \leq 0,$$
      $$2 \left(a_{13} a_{24} a_{32} a_{41}+a_{14} a_{23} a_{31} a_{42}+a_{11} a_{22} a_{34} a_{43}+a_{12} a_{21} a_{33} a_{44}\right)-\textup{perm}(A) \leq 0,$$
         $$2 \left(a_{14} a_{23} a_{32} a_{41}+a_{13} a_{24}a_{31} a_{42}+a_{12} a_{21} a_{34} a_{43}+a_{11} a_{22} a_{33} a_{44}\right)-\textup{perm}(A) \leq 0.$$
\end{remark}

Unfortunately, Corollary \ref{cor:34} does not extend beyond $n=4$. From $n=5$ onwards, the condition that $\det(\mathcal{M}(AP))\leq 0$ for all permutation matrices is stronger than having phaseless rank less than $n$, as shown in the next example.

\begin{example}\label{ex:5counterexample}
Consider the matrices $$A=\begin{bmatrix}
7 & 4 & 9 & 10 & 0\\
9 & 2 & 3 & 0 & 3\\
3 & 10 & 6 & 4 & 8\\
0 & 4 & 1 & 6 & 4\\
0 & 3 & 3 & 10 & 2
\end{bmatrix} \text{ and } P=\begin{bmatrix}
1 & 0 & 0 & 0 & 0\\
0 & 0 & 0 & 1 & 0\\
0 & 1 & 0 & 0 & 0\\
0 & 0 & 1 & 0 & 0\\
0 & 0 & 0 & 0 & 1
\end{bmatrix}.$$

We have that $\rankp(A)<5$, by Lemma \ref{lem:lop}, since no column is lopsided. However, $\det(\mathcal{M}(AP)) = 3732 > 0$, so it does not verify the determinant inequalities for all permutations matrices.
\end{example}

We now turn our attention to the boundary of the set $P^{n \times n}_{n-1}$, which we will denote by $\partial P^{n \times n}_{n-1}$. For $n \leq 4$, the explicit description we got in Corollary \ref{cor:3by3} and Corollary \ref{cor:34} immediately guarantees us that the positive part of the boundary is contained in the set of matrices $A$ such that $\det(\mathcal{M}(AP))=0$ for some permutation matrix $P$. In particular this tells us that $\partial P^{n \times n}_{n-1} \cap \RR_{++}^{n \times n} 	\subseteq S^{n \times n}_{n-1}$, for $n \leq 4$, the set of signless rank deficient matrices since $\det(\mathcal{M}(AP))=0$ implies
$\det(\mathcal{M}(AP)P^{-1})=0$ and $\mathcal{M}(AP)P^{-1}$ is simply $A$ with the signs of some entries switched. What is less clear is that exactly the same is still true for all $n$.
\begin{proposition}\label{prop:boundary}
If $A \in \partial P^{n\times n}_{n-1} \cap \mathbb{R}^{n\times n}_{++}$, then $\det(\mathcal{M}(AP))=0$ for some permutation matrix $P$.
\end{proposition}

\begin{proof}
Suppose $A \in \partial P^{n\times n}_{n-1} \cap \mathbb{R}^{n\times n}_{++}$. Since $P^{n\times n}_{n-1}$ is closed, $\rankp(A)<n$ and there must exist a sequence $A_k$ of matrices such that $A_k \rightarrow A$ and every $A_k$ is nonnegative and has phaseless rank $n$.

By Camion-Hoffman's result this implies that for every $k$ we can find a permutation matrix  $P_k \in S_n$ such that $\mathcal{M}(A_kP_k)$ is a nonsingular $M$-matrix or, equivalently, such that all eigenvalues of $\mathcal{M}(A_kP_k)$ have positive real part. Note that since there is a finite number of permutations, there exists a permutation matrix $P$ such that $P_{k_i}=P$ for an infinite subsequence $A_{k_i}$, and that $\mathcal{M}(A_{k_i}P)$ have all eigenvalues with positive real part.

Since eigenvalues vary continuously, and $\mathcal{M}(A_{k_i}P) \rightarrow \mathcal{M}(AP)$, we must have that all eigenvalues of $\mathcal{M}(AP)$ have nonnegative real part, so $\mathcal{M}(AP)$ is an $M$-matrix. It cannot be a nonsingular  $M$-matrix, as that would imply that $\rankp(A)=n$. Therefore, $\mathcal{M}(AP)$ must be singular, i.e., $\det(\mathcal{M}(AP))=0$, as intended.
\end{proof}

So, in spite of needing the smaller leading principal minors to fully describe the region, the boundary of $P^{n \times n}_{n-1}$ will still be contained in the set cut out by the determinants of the comparison matrices of the permutations of the matrices, even for $n>4$. In the next example we try to illustrate what is happening.

\begin{example}
Consider the slice of the nonnegative matrices in $\RR_+^{5 \times 5}$ that contains the identity, the all-ones matrix and the matrix in Example \ref{ex:5counterexample}, all scaled to have row sums $1$. By what we saw in Example \ref{ex:5counterexample}, we know that in this slice the set of nonnegative matrices, the set of matrices of phaseless rank less than $5$ and the set of matrices $A$ verifying $\mathcal{M}(AP) \leq 0$ for all $P$ are all distinct. This can be seen in the first image of Figure \ref{fig:5slice}, where we see the sets in light blue, green and yellow, respectively, and the three special matrices mentioned as black dots.
\begin{figure}[H]
  \centering
    \centerline{\includegraphics[width=0.47\textwidth]{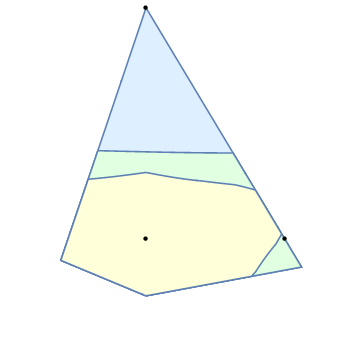} \hfill \includegraphics[width=0.47\textwidth]{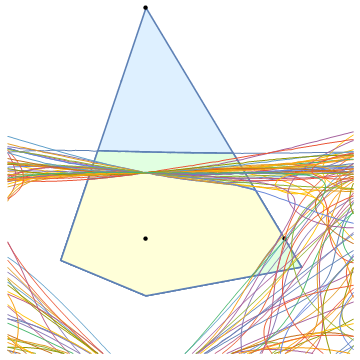}}
	\caption{A slice of the cone of $5 \times 5$ nonnegative matrices, with the nonmaximal phaseless rank region and its basic closed semialgebraic inner approximation highlighted}
	\label{fig:5slice}
\end{figure}
\end{example}
In the second image of the same figure we can see the zero sets of the $120$ different determinants of the form $\det(\mathcal{M}(AP))$ and check that the extra positive boundary points of $P^{5\times 5}_4$ do indeed come from one of them.
\subsection{Upper bounds}

In Proposition \ref{prop:rect} we have shown that for an $n \times m$ matrix, with $n \leq m$, to have phaseless rank less than $n$ it was enough to check all its $n \times n$ submatrices. A natural question is to ask if a matrix has phaseless rank less than $k$ if and only if the same is true for all its $k \times k$ submatrices, for any positive integer $k$. This is false, as was shown by Levinger (\cite{levinger1972generalization}).
\begin{theorem}[\cite{levinger1972generalization}]
Let $A = mI_n + J_n$, where $m$ is an integer with $1 \leq m < n-2$, and $I_n$ and $J_n$ are, respectively, the $n \times n$ identity and all-ones matrices. Then,
$\rankp(A) \geq m+2$.
\end{theorem}
Note that it is not hard to see that all $(m+2) \times (m+2)$ matrices of the matrix $A$ constructed above have phaseless rank at most $m+1$, so this is indeed a counterexample.

So a perfect generalization of Proposition \ref{prop:rect} is impossible, but we can try to settle for a weaker goal: discovering what having all $k \times k$ submatrices with phaseless rank less than $k$ allows us to conclude about the phaseless rank of the full matrix. This program was carried out in the same paper \cite{levinger1972generalization}, where the following result was derived.
\begin{proposition}[\cite{levinger1972generalization}]\label{prop:lev}
Let $A \in \mathbb{R}^{n\times m}_+$, with $n\leq m$. If all $k\times k$ submatrices of $A$ have nonmaximal phaseless rank, for some $k\leq n$, then $$\rankp(A)\leq m- \left\lfloor \frac{m-1}{k-1} \right\rfloor.$$
\end{proposition}
In this section we use Proposition \ref{prop:rect} to improve on this result. The result we prove is virtually the same, except that we can replace the $m$ in the bound with the smaller $n$, obtaining a much better bound for rectangular matrices.

\begin{proposition}\label{prop:prank_bd}
Let $A \in \mathbb{R}^{n\times m}_+$, with $n\leq m$. If all $k\times k$ submatrices of $A$ have nonmaximal phaseless rank, for some $k\leq n$, then $$\rankp(A)\leq n- \left\lfloor \frac{n-1}{k-1} \right\rfloor.$$
\end{proposition}

\begin{proof}
Let $M$ be an $k\times m$ submatrix of $A$. By Proposition \ref{prop:rect} the matrix $M$, has nonmaximal rank. Hence, for every $k\times m$ submatrix $M$, we can find $B_M \in \Omega(M)$ with rank less than $k$. Moreover, we are free to pick the first row of $B_M$ to be real, since scaling an entire column of $B_M$ by $e^{\theta i}$ does not change the rank or the equimodular class.

Consider then $k \times m$ submatrices $M_i$ of $A$, $i=1,\dots,\left\lfloor \frac{n-1}{k-1} \right\rfloor$ all containing the first row,which we assume non-zero, but otherwise pairwise disjoint. We can then construct a matrix $B$ by piecing together the $B_{M_i}$'s, since they coincide in the only row they share, and filling out the remaining rows, always less than $k-1$, with the corresponding entries of $A$.

By construction, in that matrix $B$ we always have in the rows corresponding to $B_{M_i}$  a row different than the first that is a linear combination of the others, and can be erased without dropping the rank of $B$. Doing this for all $i$, we get that the rank of $B$ has at least a deficiency per $B_i$, so its rank is at most
$$n- \left\lfloor \frac{n-1}{k-1} \right\rfloor,$$
and since $B$ is equimodular with $A$, $\rankp(A)$ verifies the intended inequality.
\end{proof}

Note that by setting $k=n$ we recover Proposition \ref{prop:rect}, so we have a strict extension of that result. Setting $k=2$, we get that if all $2\times 2$ minors have phaseless rank $1$ so does the matrix, which is an obvious consequence of the observation already made in Section \ref{sec:definitions} that $\rankp(A)=1$ if and only if $\rank(A)=1$. For every $k$ in-between we get new results, although not necessarily very strong. They are, however, enough to get some further geometric insight. We say that $\rankp(A)=k$ is \emph{typical} in $\RR_+^{n\times m} $ if there exists an open set in $\RR_+^{n\times m}$ for which all matrices have phaseless rank $k$.

An interesting question is the study of minimal typical ranks, which in our case corresponds to ask for the minimal $k$ for which $P^{n\times m}_k$ has full dimension.
We claim that if $k$ is typical, then we must have $k \geq \left\lceil \frac{n+m-\sqrt{(n-1)^2+(m-1)^2}}{2} \right\rceil$.
Take the map which sends each matrix in $(\mathbb{C}^{*})^{n\times m}$ to its entrywise absolute value, in $\RR_{++}^{n\times m}$. The image under this map of the variety of complex matrices with no zero entries and of rank at most $k$ is $P^{n\times m}_{k}\cap \RR_{++}^{n\times m}$, which is full-dimensional if and only if $k$ is typical. Note that we can assume that every matrix in the domain has real entries in the first row and column, since row and column scaling by complex numbers of absolute value one preserve both the rank and the entrywise absolute value matrix. The real dimension of the variety of complex matrices of rank at most $k$ with real first row and column  is $2(n+m-k)k$, twice the number of complex degrees of freedom, minus $m+n-1$, the numbers of entries forced to be real. This difference should be at least $n\times m$, the dimension of $P^{n\times m}_{k}\cap \RR_{++}^{n\times m}$, since the map is differentiable. Thus, we must have
$$2(n+m-k)k-n-m+1 \geq nm,$$
which boils down to $$k\geq \left\lceil \frac{n+m-\sqrt{(n-1)^2+(m-1)^2}}{2} \right\rceil,$$ because $k$ is a positive integer.

\begin{corollary}\label{cor:dimension}
For $\RR_+^{n \times m}$, with $3 \leq n \leq m$, the minimal typical phaseless rank $k$ must verify
$$ \left\lceil \frac{n+m-\sqrt{(n-1)^2+(m-1)^2}}{2} \right\rceil \leq k \leq \left\lceil \frac{n+1}{2} \right\rceil.$$
\end{corollary}
\begin{proof}
The lower bound comes from the above dimension count. To prove the upper bound, note that the $3 \times 3$ all-ones matrix has phaseless rank $1$ (less than three), and any small enough entrywise perturbation of it also has phaseless rank less than $3$, since it will still have nonlopsided columns. This means that the $n \times m$ all-ones matrix, and any sufficiently small perturbation of it, have all $3 \times 3$ submatrices with nonmaximal phaseless rank, which implies, by Proposition \ref{prop:prank_bd}, that their phaseless rank is at most $\left\lceil \frac{n+1}{2} \right\rceil$. Hence, there exists an open set of $\RR_+^{n \times m}$ in which every matrix has phaseless rank less or equal than that number, which implies the smallest typical rank is at most that, giving us the upper bound.
\end{proof}

For $m$ much larger than $n$ the bound is almost tight, since the lower bound converges to $n/2$. In fact, for odd $n$ and sufficiently large $m$ we will have that the typical rank is actually $\frac{n+1}{2}$, since that will be the only integer satisfying both bounds.

\section{Applications and outlook}

\subsection{The amoeba point of view}\label{subsec:amb_pov}

Many of the results developed in the previous sections have nice interpretations from the viewpoint of amoeba theory. Here, we will introduce some concepts and problems coming from this area of research and show the implications of the work previously developed.

As mentioned before, checking for amoeba membership is a hard problem. Even certifying that a point is not in an amoeba is generally difficult. To that end, several necessary conditions for amoeba membership have been developed. One such condition is the non-lopsidedness criterion. In its most basic form, this gives a necessary condition for a point to be in the amoeba of the principal ideal generated by some polynomial $f$, $\mathcal{A}(f)$.

Let $f \in \mathbb{C}[z_1,\ldots,z_n]$ and $a \in \mathbb{R}^n$. By writing $f$ as a sum of monomials, $f(\mathbf{z})=m_1(\mathbf{z})+\ldots+m_d(\mathbf{z})$, define
$$f\{\mathbf{a}\}:=\{|m_1(\mathbf{a})|,\ldots,|m_d(\mathbf{a})|\}.$$
It is clear that in order for $\mathbf{a}$ to be the vector of absolute values of some complex root of $f$, the vector $f\{\mathbf{a}\}$ cannot be lopsided, as it must cancel after the phases are added in. We then define
$$\textup{Nlop}(f)=\{a \in \RR^n \ : \ f\{\mathbf{a}\} \text{ is not lopsided}\}.$$
It is clear that $\mathcal{A}(f) \subseteq \textup{Log}(\textup{Nlop}(f))$, but the inclusion is generally strict. One immediate consequence of Example \ref{ex:3x3characterization} is the following.
\begin{proposition}
Let $f=det(X)$ be the cubic polynomial in variables $x_{ij}$, $i,j=1,2,3$. Then
$$\mathcal{A}(f) = \textup{Log}(\textup{Nlop}(f)).$$
\end{proposition}
So, the above proposition gives us an example where nonlopsidedness is a necessary and sufficient condition. In fact, this is just a special case of a more general result from amoeba theory: that for any polynomial whose support forms the set of vertices of a simplex (which is the case for the $3\times3$ determinant), it holds that $\mathcal{A}(f) = \textup{Log}(\textup{Nlop}(f))$ . This follows from \cite{forsberg2000laurent} (see, for instance, Theorem 3.1 of \cite{theobald2013amoebas} for details).

Another interesting example that we can extract from our results concerns amoeba bases. Purbhoo shows, in \cite{purbhoo2008nullstellensatz}, that the amoebas of general ideals can be reduced in a way to the case of principal ideals, since $\mathcal{A}(V(I))=\bigcap_{f \in I} \mathcal{A}(f)$. The problem is that this is an infinite intersection, which immediately raises the question if a finite intersection may suffice. This suggests the notion of an \emph{amoeba basis}, introduced in \cite{schroeter2013boundary}.
\begin{definition}
Given an ideal $I \subseteq \mathbb{C}[z_1,\ldots,z_n]$, we call a finite set $B \subset I$ an amoeba basis for $I$ if it generates $I$ and it verifies the property
$$\mathcal{A}(V(I))=\bigcap_{f \in B} \mathcal{A}(f)$$
while any proper subset of $B$ does not.
\end{definition}
Unfortunately, amoeba bases may fail to exist and in fact very few examples of them are known. In \cite{nisse2018describing} it is proved that varieties of  a particular kind, those that are \emph{independent complete intersections}, have amoeba bases, and it is conjectured that only union of those can have them (see \cite[Conjecture 5.3]{nisse2018describing}). Proposition \ref{prop:rect} gives us a nice new example of such nice behavior, disproving the conjecture, since the variety of $n \times m$ rectangular matrices, with $n < m$, of rank less than $n$ is irreducible and not even a set-theoretic complete intersection \cite{Bruns90}.
\begin{corollary}
Let $X$ be an $n \times m$ matrix of indeterminates. The set of maximal minors of $X$ is an amoeba basis for the determinantal ideal they generate.
\end{corollary}
Note that this is just another result in a long line of results about the special properties of the basis of maximal minors of a matrix of indeterminates, notoriously including the fact that they form a universal Groebner basis, as proved in \cite{Bernstein1993}.
For $3 \times n$ matrices we actually have that the nonlopsidedness of the generators is enough to guarantee the amoeba membership, an even stronger condition.

All other results automatically translate to amoeba theory, and some have interesting translations. We provide explicit semialgebraic descriptions for the amoeba of maximal minors, adding one example to the short list of amoebas for which such is available, as pointed out in \cite[Question 3.7]{nisse2018describing}. Moreover, Proposition \ref{prop:boundary} implies that the boundary of the amoeba of the determinant of a square matrix of indeterminates is contained in the image by the entrywise absolute value map of the set of its real zeros, while Corollary \ref{cor:dimension} states some conditions for full dimensionality of the amoeba of the variety of bounded rank matrices.

\subsection{Implications on semidefinite rank}

As we saw before, upper bounds on the phaseless rank will immediately give us upper bounds on the complex semidefinite rank. One can use that to improve on some results in the literature, and hopefully to construct examples.

For a simple illustration, recall the following result proved in \cite{lee2017some}, that gives sufficient conditions for nonmaximality of the complex semidefinite rank of a matrix.

\begin{proposition}[\cite{lee2017some}]\label{prop:leewei}
Let $A\in \mathbb{R}^{n\times m}_+$. If every column of $\sqrt[\circ]{A}$ has no dominant entry (i.e., if every column of $\sqrt[\circ]{A}$ is not lopsided), then $\rankpsd(A)<n$.
\end{proposition}

We remark that the assumption in the previous result is just a sufficient condition for $\rankp(\sqrt[\circ]{A}) < n$, which implies $\rankpsd(A)<n$, by Proposition \ref{prop:ineqs}. This observation easily follows from applying Lemma \ref{lem:lop} to $\sqrt[\circ]{A}$. This means that Proposition \ref{prop:leewei} is just a specialization of the following more general statement.

\begin{proposition}
Let $A\in \mathbb{R}^{n\times m}_+$. If $\rankp(\sqrt[\circ]{A}) < n$, then $\rankpsd(A)<n$.
\end{proposition}
\label{eqn:ranks}

One can check whether $\rankp(\sqrt[\circ]{A}) < n$ by using both Proposition \ref{prop:rect}, if the matrix is not square, and Theorem \ref{thm:cam_hof}. More generally, Proposition \ref{prop:ineqs} dictates that every upper bound for $\rankp(\sqrt[\circ]{A})$ is an upper bound for $\rankpsd(A)$. Thus, we have the following corollary of Proposition \ref{prop:prank_bd}.
\begin{corollary}\label{cor:sqrank}
Let $A \in \mathbb{R}^{n\times m}_+$, with $n \leq m$. If all $k\times k$ submatrices of $\sqrt[\circ]{A}$ have nonmaximal phaseless rank, $$\rankpsd(A)\leq n-\left\lfloor \frac{n-1}{k-1}\right\rfloor.$$
\end{corollary}

One can actually improve on both these results by removing the need to consider the Hadamard square root. To do that, we need an auxiliary lemma, concerning the Hadamard product of matrices:

\begin{lemma}
Let $A \in \mathbb{R}^{n\times n}_+$ and $\alpha \geq 1$. If $\rankp(A)=n$, then $\rankp(A^{\circ \alpha})=n$, where $A^{\circ \alpha}$ is obtained from $A$ by taking entrywise powers $\alpha$.
\end{lemma}
\begin{proof}
By Theorem \ref{thm:cam_hof}, $\rankp(A)=n$ if and only if there exists a permutation matrix $P$ such that $\mathcal{M}(AP)$ is a nonsingular M-matrix, which is equivalent to saying that the minimum real eigenvalue of $\mathcal{M}(AP)$ is positive, according to Proposition \ref{prop:mmatrix}, i.e., $\sigma(AP)>0$.

But then, Theorem $4$ from \cite{elsner1988perron} guarantees precisely that we must have $$\sigma(A^{\circ \alpha}P) = \sigma((AP)^{\circ  \alpha}) \geq \sigma(AP)^{\alpha} > 0,$$
proving that $\rankp(A^{\circ \alpha})=n$.
\end{proof}

By specializing $\alpha=2$ and applying the previous Lemma to the Hadamard square root of $A$ we get the following immediate Corollary.

\begin{corollary}
Let $A \in \mathbb{R}^{n\times n}_+$. If $\rankp(A)<n$, $\rankp(\sqrt[\circ]{A})<n.$
\end{corollary}

This can be used to get a simpler upper bound on the complex semidefinite rank, testing submatrices of $A$ instead of its square root.

\begin{corollary}\label{cor:sqrank_psd}
Let $A \in \mathbb{R}^{n\times m}_+$, with $n \leq m$. If all $k\times k$ submatrices of $A$ have nonmaximal phaseless rank, $$\rankpsd(A)\leq n-\left\lfloor \frac{n-1}{k-1}\right\rfloor.$$
\end{corollary}

This can be used to derive simple upper bounds on the extension complexity of polytopes. Recall that for a $d$-dimensional polytope, $P$,
its slack matrix, $S_P$, has rank $d+1$ and its complex semidefinite rank is the complex semidefinite extension complexity of $P$. Since every $(d+2)\times (d+2)$ submatrix of $S_P$ has rank $d+1$, it also has phaseless rank at most $d+1$. Thus, by applying the previous corollary we obtain the following result.
\begin{corollary}\label{cor:sqrank_pol}
Let $P$ be a $d$-dimensional polytope with $v$ vertices and $f$ facets, and $m=\min\{v,f\}$ then
$$\rankpsd(S_P) \leq m-\left\lfloor \frac{m-1}{d+1}\right\rfloor.$$
\end{corollary}
For $d=2$, for example, this gives us an upper bound of $\left \lceil \frac{2n+1}{3} \right \rceil$ for the complex extension complexity of an $n$-gon, which is similar asymptotically to the $4 \left \lceil \frac{n}{6} \right \rceil$ bound derived in \cite{gouveia2015worst} and slightly
better for small $n$ (note that that bound is valid for the real semidefinite extension complexity, and so automatically for the complex case too). Of course it is just linear, so it does not reach the sublinear complexity proved by Shitov in \cite{shitov2014sublinear} even for the linear extension complexity, but it is applicable in general and can be useful for small polytopes in small dimensions. Moreover, it is, as far as we know, the only non-trivial bound that works for polytopes of arbitrary dimension. As a last remark, we note that such lift can explicitly can be constructed. This can easily be done from an actual rank 
$m-\left\lfloor \frac{m-1}{d+1}\right\rfloor$ matrix that is equimodular to the Hadamard square root of the slack matrix, and such matrix can, with a small amount of work, be explicitly constructed from our results.
\subsection{Equiangular lines}

A set of $n$ lines in the vector space $\mathbb{R}^d$ or $\mathbb{C}^d$ is called equiangular if all the lines intersect at a single point and every pair of lines makes the same angle. Bounding the maximum number of real equiangular lines for a given dimension has long been a popular research problem. Classically, we want bounds on the absolute maximum number of such lines (denoted by $N(d)$) or on the maximum number for a given common angle $\arccos(\alpha)$ (denoted by $N_{\alpha}(d)$). A somewhat thorough survey on this type of results can be found in  \cite{de2018k}, while further reading on the real case can be seen in \cite{greaves2016equiangular}, \cite{jiang2019equiangular}, and \cite{lemmens1991equiangular}. 

The complex case has seen a flurry of recent developments due to its connection to quantum physics (see for instance \cite{MR2142983},\cite{MR2301093},\cite{MR2059685},\cite{MR2662471}). In fact, it is well known that the maximum number of complex equiangular lines in $\mathbb{C}^d$, denoted by $N^{\mathbb{C}}(d)$, is bounded from above by $d^2$ and it is conjectured that $N^{\mathbb{C}}(d)=d^2$ for all $d\geq 2$ (\cite{zauner1999grundzuge}). When such a maximum set of $d^2$ lines exists, one can construct a symmetric, informationally complete, positive operator-valued measure (SIC-POVM), an object that plays an important role in quantum information theory. Recent developments in the construction of large sets of complex equiangular lines can be found in \cite{jedwab2015large} and  \cite{jedwab2015simple}.

To see how these notions relate to the object of our study, consider a set of $n$ lines through a point in a $d$-dimensional Euclidean space, which we consider either $\mathbb{R}^d$ or $\mathbb{C}^d$. Let $v_i$, $i=1,...,n$, be unit vectors for each of the lines, and let $V$ be the matrix whose columns correspond to these vectors. Note that the lines having pairwise angle $\arccos(\alpha)$ is the same as having $|v_i^*v_j|= \alpha$ for all $i \not = j$. More precisely, $|V^* V| = A^{\alpha}_n$, where $A^{\alpha}_n$ denotes the $n\times n$ matrix with ones on the diagonal and $\alpha$'s everywhere else, which means $A^{\alpha}_n$ is equimodular to a positive semidefinite matrix of rank at most $d$. Conversely, if $A^{\alpha}_n$ is equimodular to a positive semidefinite matrix of rank at most $d$, one can do an eigendecomposition to attain a set of $n$ equiangular lines in the $d$-dimensional Euclidean with common angle $\arccos(\alpha)$. This immediately suggests a semidefinite variant of the phaseless rank.

\begin{definition}
For a symmetric matrix $A \in \mathbb{R}^{n\times n}_+$, its psd-phaseless rank is defined as
$$\rankp^{\text{psd}}(A) = \min\{\rank(B): B\in \Omega(A) \text{ and }B \succeq 0\}.$$
\end{definition}

We can then use this notion to highlight that the problem of finding equiangular lines with fixed angle is equivalent to that of finding a matrix rank.

\begin{proposition}
For $0\leq \alpha\leq 1$, $\rankp^{\text{psd}}(A^{\alpha}_n)$  is the smallest dimension $d$ for which there exists an equiangular set of $n$ lines in $\mathbb{C}^d$  with common angle $\arccos{\alpha}$.
\end{proposition}

Note that, in particular, $\rankp^{\text{psd}}(A) \geq \rankp(A)$, so lower bounds on the usual phaseless rank give us upper bounds on the numbers of equiangular lines. In the real case, we can introduce the analogous notion of psd-signless rank and, in that case, the trivial signless rank inequality from Lemma \ref{lem:inequality} recovers the traditional Gerzon upper bound for the number of equiangular lines. In the complex case, the inequality $N^{\mathbb{C}}(d)\leq d^2$ can be rewritten as $\rankp^{\text{psd}}(A^{\alpha}_n)\geq \sqrt{n}$ for all $\alpha$ which, once again, follows directly from  $\rankp(A^{\alpha}_n)$ being a lower bound for $\rankp^{\text{psd}}(A^{\alpha}_n)$ and Lemma \ref{lem:inequality}. To illustrate this strategy of turning lower bounds on phaseless rank into upper bounds on the number of equiangular lines,  we present a simple result derived from our basic bounds on phaseless rank.

\begin{proposition}
For $\alpha<\frac{1}{d}$, $N_{\alpha}^{\mathbb{C}}(d)=d$.
\end{proposition}
\begin{proof}
Fix $d$ and let $\alpha<\frac{1}{d}$. Observe that one can write $N_{\alpha}^{\mathbb{C}}(d)$ as
$$\max\{n:\rankp^{\text{psd}}(A^{\alpha}_n)\leq d\}$$
Since $A^{\alpha}_{d+1}$ has lopsided columns, and is a submatrix of any $A^{\alpha}_n$ for $n>d$, we have  $$\rankp^{\text{psd}}(A^{\alpha}_{n}) \geq \rankp(A^{\alpha}_n) \geq d+1$$
for any $n > d$. Since $A^{\alpha}_d$ is positive semidefinite and has rank $d$, the result follows.
\end{proof}

While fairly simple, this result highlights the usefulness of deriving effective lower bounds to the phaseless rank, as a means to obtain upper bounds to $N_{\alpha}^{\mathbb{C}}(d)$. A related classical concept that can be studied in terms of psd-phaseless rank is that of mutually unbiased bases in $\mathbb{C}^d$ (MUB's). Two orthonormal bases $\{u_1,...,u_d\}$ and $\{v_1,...,v_d\}$ of $\mathbb{C}^d$ are said to be unbiased if $|u_i^*v_j|=\frac{1}{\sqrt{d}}$ for all $i$ and $j$. A set of orthonormal bases is a set of mutually unbiased bases if all pairs of distinct bases are unbiased. It is known that there cannot exist sets of more than $d+1$ MUB's in $\mathbb{C}^d$, and such sets exist for $d$ a prime power, but the precise maximum number is unknown even for $d=6$, where it is believed to be three (see \cite{durt2010mutually}, \cite{bandyopadhyay2002new} and  \cite{brierley2009constructing} for more information and a survey into this rich research area). To translate this in terms of phaseless rank, consider the matrix $B_d^k$ defined as the matrix of $k \times k$ blocks where the blocks in the diagonal are $d \times d$ identities and the off-diagonal ones are constantly equal to $\frac{1}{\sqrt{d}}$. The following simple fact is then clear.

\begin{proposition}
There exists a set of $k$  mutually unbiased bases in $\mathbb{C}^d$ if and only if $\rankp^{\text{psd}}(B_d^k) = d$.
\end{proposition}

As in equiangular lines, lower bounds on the phaseless rank have the potential to give upper bounds on the maximum number of MUB's.

\subsection{Conclusion and some open questions}

Throughout this paper we established the connection between the classical results of Camion and Hoffman on equimodular classes of matrices with the modern developments in the theories of amoebas and semidefinite extension complexity.
This provided a rich field of motivation and applications, and allowed for interesting and new developments. However, many questions remain completely open and are ripe for further explorations.

\begin{enumerate}
\item Is it possible to characterize other cases besides the nonmaximal phaseless rank? The simplest outstanding case would be to characterize $4\times 4$ matrices of phaseless rank at most $2$.

\item Since the phaseless rank has strong conceptual connections to both the rank minimization and the phase retrieval problems can one use the body of work on approximations to those problems to develop some approximations to these quantities?

\item What can we say about the complexity of computing the phaseless rank?

\item While some work was already carried out here on the dimension of these semialgebraic sets, it should be possible to state more precise results on which values of the phaseless rank are typical.
\end{enumerate}

\section*{Acknowledgments}
The authors would like to thank Ant\'{o}nio Leal Duarte for pointing us towards the literature on Camion-Hoffman's Theorem, and Timo de Wolff for the encouragement and constructive feedback on the amoeba applications.

\bibliographystyle{plain}
\bibliography{bibliografia}

\end{document}